\documentclass[mn,fleqn]{w-art}
\usepackage{times}
\usepackage{w-thm}

\usepackage{amsmath, amsthm, amscd, amsfonts, amssymb, graphicx, color}
\usepackage{mathrsfs}

\usepackage[OT2,OT1]{fontenc}
\newcommand\cyr
{
\renewcommand\rmdefault{wncyr}
\renewcommand\sfdefault{wncyss}
\renewcommand\encodingdefault{OT2}
\normalfont
\selectfont
}
\DeclareTextFontCommand{\textcyr}{\cyr}

\usepackage[]{graphicx}
\begin{document}
\DOIsuffix{mana.DOIsuffix}
\Volume{248}
\Month{01}
\Year{2005}
\pagespan{1}{}
\Receiveddate{15 November 2005}
\Reviseddate{30 November 2005}
\Accepteddate{2 December 2005}
\Dateposted{3 December 2005}



\title[Arakelov geometry of Cuntz-Pimsner algebras]{Arakelov geometry of Cuntz-Pimsner algebras}

\author[I. V. Nikolaev]{Igor V. Nikolaev\footnote{Corresponding
     author: e-mail: {\sf igor.v.nikolaev@gmail.com}}\inst{1}} 
     \address[\inst{1}]{Department of Mathematics and Computer Science, St.~John's University, 8000 Utopia Parkway,  
                                 New York,  NY 11439, United States.}

\dedicatory{All data are available as part of the manuscript}
\begin{abstract}
We use $K$-theory of the $C^*$-algebras to study 
the Arakelov geometry, i.e.  a compactification of
the arithmetic schemes $V\to \operatorname{Spec} ~\mathbf{Z}$.  
In particular, it is proved that the Picard group of $V$
is isomorphic to the   $K_0$-group of a Cuntz-Pimsner algebra associated to $V$. 
We apply this result to the finiteness problem for the algebraic varieties over  number 
fields. 
\end{abstract}

\maketitle                   


\section{Introduction}
\subsection{Arakelov geometry}
The Arakelov theory deals  with
a compactification problem for the arithmetic schemes. 
We refer the reader to the original work \cite{Ara1}  and  
 \cite{S} for an introduction. 
To outline the idea, consider an arithmetic variety $V$ given by a system of 
the polynomial equations
\begin{equation}\label{eq1.1}
f_1(x_1,\dots,x_n)=f_2(x_1,\dots,x_n)=\dots=f_k(x_1,\dots,x_n)=0,
\end{equation}
where $f_1,\dots,f_k\in\mathbf{Z}[x_1,\dots,x_n]$ are homogeneous 
polynomials with integer coefficients. 
If  (\ref{eq1.1}) has a solution in  $\mathbf{Z}$,
one can always reduce it modulo prime $p$ 
thus obtaining a solution  in the finite
field $\mathbf{F}_p$.  This phenomenon has a remarkable 
description in terms of  commutative algebra 
and Grothendieck's theory of schemes. 
Namely, one can think of (\ref{eq1.1}) as a 
flat fiber bundle  $V\to\operatorname{Spec}~\mathbf{Z}$
with the  fibers $V(\mathbf{F}_p)$.    
The well known analogy between function fields and number fields 
says that $\operatorname{Spec}~\mathbf{Z}$  can be complemented by  a prime ideal at 
the point $p=\infty$. 
It seems to be difficult to define an ideal satisfying these conditions.
The problem was settled by Arakelov  when  $k=1$ and $n=2$, 
but his method also works   for all  $V\to\operatorname{Spec}~\mathbf{Z}$
of relative dimension $1$.
Specifically, it was proved that  a Riemann surface $X(\mathbf{C})$ endowed with a hermitian 
metric on the holomorphic vector bundle over  $X(\mathbf{C})$ 
corresponds to a fiber at infinity  \cite{Ara1}.
At the heart of Arakelov's construction is the notion of a 
compactified principal divisor:
\begin{equation}\label{eq1.2}
(\varphi)_c=(\varphi)+  v_{\infty}(\varphi) X(\mathbf{C}),
\end{equation}
where $\varphi$ is a rational function on $V$, $(\varphi)$ is the
principal divisor of $\varphi$ and  $v_{\infty}(\varphi)=-\int_{X(\mathbf{C})}\log |\varphi| ~d\mu$
so that  the measure $\mu$ is defined by a Hermitian metric on the holomorphic vector bundle over  
$X(\mathbf{C})$. 
In particular, one can define a compactified Picard group $\operatorname{Pic}_c(V)$ by taking the quotient of the 
abelian group of formal compactified divisors by its subgroup of principal divisors (\ref{eq1.2}). 
 Arakelov's theory has been generalized  to the 
higher dimensions in  \cite{S}.

\subsection{Cuntz-Pimsner algebras}
 The Cuntz-Krieger algebra $\mathcal{O}_A$ is a  $C^*$-algebra
generated by the  partial isometries $s_1,\dots, s_n$ which satisfy  the relations
$s_i^*s_i=\sum_{j=1}^n a_{ij}s_js_j^*$,
where $A=(a_{ij})$ is a square matrix with the integer entries  $a_{ij}\in \{0, 1, 2, \dots \}$
\cite{CunKrie1}.  
Such algebras  appear naturally in the study of local factors \cite{Nik1}. 
The Cuntz-Pimsner algebra is a generalization
of  $\mathcal{O}_A$ to the countably infinite matrices $A_{\infty}\in\operatorname{GL}_{\infty}(\mathbf{Z})$  
with the integer entries
\cite{PasRae1}.
Recall that the matrix $A_{\infty}$ is called row-finite,  if for each $i\in\mathbf{N}$
the number of $j\in\mathbf{N}$ with $a_{ij}\ne 0$ is finite.  The matrix $A$ is 
said to be irreducible, if some power of $A$ is a strictly positive matrix and $A$ is not a
permutation matrix; we refer the reader to Section 2.2 for the details and 
\cite[Section 2]{PasRae1} for a set of examples. 
 It is known that if  $A_{\infty}$ is row-finite and irreducible, then the 
Cuntz-Pimsner algebra 
 $\mathcal{O}_{A_{\infty}}$ is a well-defined  and simple 
\cite[Theorem 1]{PasRae1}.
The $K$-theory of the $C^*$-algebra  $\mathcal{O}_{A_{\infty}}$ is calculated  by Theorem \ref{thm2.2},
see also  \cite[Theorem 3]{PasRae1}.

\subsection{Local factors at infinity}
We recall  an interplay between the Cuntz-Pimsner algebras and
the local factors at infinity \cite{Ser1}.  Let  $char ~A_{\infty}:=\det (A_{\infty}-sI)$ be  the characteristic polynomial of $A_{\infty}\in\operatorname{GL}_{\infty}(\mathbf{Z})$,
where $\det$ is a determinant introduced in  \cite[Section 1]{Den1}.
 The  polynomial $char ~A_{\infty}$ of variable $s$ is well defined by the row-finiteness of matrix $A_{\infty}$  
 \cite[Note 2.1.11]{PasRae1}.
 For every smooth  $n$-dimensional 
projective variety $V$ over a number field $k$
there exist  the Cuntz-Pimsner algebras $\mathcal{O}_{A^i_{\infty}}$,  
such that the Hasse-Weil zeta function of $V$  \cite{Ser1}
is given by the formula \cite[Theorem 1.1]{Nik1}:
\begin{equation}\label{eq1.4}
Z_V(s)=\prod_{i=0}^{2n} \left( char ~A_{\infty}^i \right)^{(-1)^{i+1}}.
\end{equation}

\subsection{Main result}
The aim of our note is an isomorphism between the $K_0$-group of the Cuntz-Pimsner 
algebra $\mathcal{O}_{A^1_{\infty}}$ and the  group $\operatorname{Pic}_c(V)$  (Theorem \ref{thm1.1}).  Such a result is applied
 to the finiteness problem for the algebraic varieties over  number 
fields (Conjecture \ref{cnj4.1}). 
Namely,  let $\mathcal{O}_{A^1_{\infty}}$  be the Cuntz-Pimsner 
algebra defined  by matrix $A^1_{\infty}\in\operatorname{GL}_{\infty}(\mathbf{Z})$ which satisfies equation (\ref{eq1.4}) 
for a variety $V$. Our main result can be 
formulated as follows. 
\begin{theorem}\label{thm1.1}
There exists a canonical isomorphism of the abelian groups 
\linebreak
$K_0(\mathcal{O}_{A^1_{\infty}})\cong\operatorname{Pic}_c(V)$, 
where $\operatorname{Pic}_c(V)$ is Arakelov's compactification of the 
Picard group  of variety $V$. 
\end{theorem}
The paper is organized as follows.  A brief review of the preliminary facts is 
given in Section 2. Theorem \ref{thm1.1}
proved in Section 3.  
An application of Theorem \ref{thm1.1} is discussed 
in Section 4.

\section{Preliminaries}
We briefly review Arakelov theory,  Cuntz-Pimsner algebras  and local factors at infinity. 
We refer the reader to    \cite{Ara1},  \cite{Nik1},     
\cite{PasRae1}
and    \cite{Ser1} for a detailed exposition.

\subsection{Arakelov geometry}
Let $k$ be a number field and let $\Lambda\subset k$ be its ring of algebraic integers. 
Denote by $m=\deg~(k|\mathbf{Q})$ the degree of the field $k$ over the rationals $\mathbf{Q}$.
Let $X$ be an algebraic variety over $k$  and $V\to\operatorname{Spec}~\Lambda$ the corresponding 
arithmetic scheme. 
A finite divisor is the usual divisor of $V$, i.e. $D_{fin}=\sum k_iC_i$,
where $k_i\in\mathbf{Z}$ and $C_i\subset V$ the irreducible closed subsets of $V$
of codimension $1$. 
By $X_{\infty}$ one understands an analog of fiber of the bundle $V\to\operatorname{Spec}~\Lambda$
over the infinite point of $Spec ~\Lambda$.  
Along with the finite divisors, one considers divisors of the form:
\begin{equation}\label{eq2.1}
D=D_{fin}+\sum_{i=1}^m \lambda_{\infty}^{(i)} X_{\infty}^{(i)},
\end{equation}
where $\lambda_{\infty}^{(i)}\in\mathbf{R}$ and summation is taken over all embeddings of the field $k$. 
To define a (compactified) principal divisor, it is necessary to fix on each $X_{\infty}^{(i)}$ a hermitian metric,
such that the corresponding volume elements satisfy the condition $\int_{X_{\infty}^{(i)}}d\mu=1$.  
(The construction can be generalized to the higher dimensions using hermitian metric on holomorphic 
vector bundle over $X_{\infty}^{(i)}$ which are required to be invariant under the complex conjugation.)
Let $\varphi$ be a rational function on $V$ and $(\varphi)$ be the principal divisor 
corresponding to $\varphi$. One defines  $v_{\infty}(\varphi)=-\int_{X_{\infty}^{(i)}}\log |\varphi|d\mu$.
The compactified principal divisor of $\varphi$  is given by the formula:
\begin{equation}\label{eq2.2}
(\varphi)_c=(\varphi)+\sum_{i=1}^m v_{\infty}(\varphi)X_{\infty}^{(i)}. 
\end{equation}
Divisors (\ref{eq2.1}) and principal divisors (\ref{eq2.2}) make an additive abelian group $\operatorname{Div}(V)$ and $P(V)$, respectively. 
Since $P(V)\subset \operatorname{Div}(V)$, one can introduce a compactified Picard group $\operatorname{Pic}_c(V)=\operatorname{Div}(V)/P(V)$. 
On the other hand, the well known product formula of  valuation theory gives rise to the Picard group 
$\operatorname{Pic}_c(k)$ of the number field $k$.
Roughly speaking, the following result says that the Arakelov geometry is a correct compactification of the arithmetic schemes.
\begin{theorem} \label{thm2.1}
{\bf (\cite[Section 0]{Ara1})}
Arakelov completion of the fiber map $V\to\operatorname{Spec}~\Lambda$  induces a homomorphism  of the abelian  groups $\operatorname{\operatorname{Pic}}_c(k)\to \operatorname{Pic}_c(V)$. 
\end{theorem}

\subsection{Cuntz-Pimsner algebras}
The Cuntz-Krieger algebra $\mathcal{O}_A$ is a  $C^*$-algebra
generated by the  partial isometries $s_1,\dots, s_n$ which satisfy  the relations
\begin{equation}
\left\{
\begin{array}{ccc}
s_1^*s_1 &=& a_{11} s_1s_1^*+a_{12} s_2s_2^*+\dots+a_{1n}s_ns_n^*\\ 
s_2^*s_2 &=& a_{21} s_1s_1^*+a_{22} s_2s_2^*+\dots+a_{2n}s_ns_n^*\\ 
                  &\dots&\\
s_n^*s_n &=& a_{n1} s_1s_1^*+a_{n2} s_2s_2^*+\dots+a_{nn}s_ns_n^*,             
\end{array}
\right.
\end{equation}
where $A=(a_{ij})$ is a square matrix with  $a_{ij}\in \{0, 1, 2, \dots \}$. 
(Note that the original definition of $\mathcal{O}_A$ says that  $a_{ij}\in \{0, 1\}$
but is known to be extendable to all non-negative integers \cite{CunKrie1}.)   
Such algebras  appear naturally in the study of local factors \cite{Nik1}.

The Cuntz-Pimsner algebra  $\mathcal{O}_{A_{\infty}}$ corresponds to the case of 
 the countably infinite matrices $A_{\infty}\in\operatorname{GL}_{\infty}(\mathbf{Z})$  
\cite{PasRae1}.
The matrix $A_{\infty}$ is called row-finite,  if for each $i\in\mathbf{N}$
the number of $j\in\mathbf{N}$ with $a_{ij}\ne 0$ is finite.  The matrix $A_{\infty}$ is 
said to be irreducible, if some power of $A_{\infty}$ is a strictly positive matrix and $A_{\infty}$ is not a
permutation matrix.  If  $A_{\infty}$ is row-finite and irreducible, then the 
Cuntz-Pimsner algebra  $\mathcal{O}_{A_{\infty}}$ is a well-defined  and simple 
 \cite[Theorem 1]{PasRae1}.

An AF-core $\mathscr{F}\subset \mathcal{O}_{A_{\infty}}$ is an Approximately Finite (AF-) $C^*$-algebra
defined by the closure of  the infinite union $\cup_{k,j} \cup_{i\in V_k^j} \mathscr{F}_k^j(i)$,
where  $\mathscr{F}_k^j(i)$ are finite-dimensional $C^*$-algebras 
built from matrix $A_{\infty}$  \cite[Definition 2.2.1]{PasRae1}. 
Let $\alpha: \mathcal{O}_{A_{\infty}}\to \mathcal{O}_{A_{\infty}}$ be an automorphism 
acting on the generators $s_i$ of $\mathcal{O}_{A_{\infty}}$ by
to the formula $\alpha_z(s_i)=zs_i$, where $z$ is a complex number  $|z|=1$.
One gets an action of the abelian group $\mathbb{T}\cong\mathbf{R}/\mathbf{Z}$ on  $\mathcal{O}_{A_{\infty}}$. 
The Takai duality \cite[p. 432]{PasRae1} says that:
\begin{equation}\label{eq2.4} 
\mathscr{F}\rtimes_{\hat\alpha}\mathbb{T}\cong \mathcal{O}_{A_{\infty}}\otimes\mathcal{K},
\end{equation}
where $\hat\alpha$ is the Takai dual of $\alpha$ and $\mathcal{K}$ is the
$C^*$-algebra of compact operators.  Using (\ref{eq2.4}) one can calculate the the 
$K$-theory of  $\mathcal{O}_{A_{\infty}}$. Namely, the following statement is true. 
\begin{theorem}\label{thm2.2}
{\bf (\cite[Theorem 3]{PasRae1})}
If $A_{\infty}$ is row-finite irreducible matrix, then there exists an exact 
sequence of the abelian groups:
\begin{equation}\label{eq2.5} 
0\to K_1(\mathcal{O}_{A_{\infty}})\to \mathbf{Z}^{\infty}\buildrel 1-A_{\infty}^t\over\longrightarrow 
 \mathbf{Z}^{\infty}\buildrel i_*\over\longrightarrow  K_0(\mathcal{O}_{A_{\infty}})\to 0, 
\end{equation}
so that $K_0(\mathcal{O}_{A_{\infty}})\cong  \mathbf{Z}^{\infty}/(1-A_{\infty}^t) \mathbf{Z}^{\infty}$ and 
 $K_1(\mathcal{O}_{A_{\infty}})\cong Ker ~(1-A_{\infty}^t)$, where  $A_{\infty}^t$ is the transpose 
 of  $A_{\infty}$ and $i: \mathscr{F}\hookrightarrow \mathcal{O}_{A_{\infty}}$.
 Moreover, the Grothendieck semigroup $K_0^+(\mathscr{F})\cong \varinjlim (\mathbf{Z}^{\infty}, A_{\infty}^t)$. 
 \end{theorem}

\subsection{Local factors at infinity}
Let $V$ be an $n$-dimensional smooth projective variety over a number field $k$
and let $V(\mathbf{F}_q)$ be a good reduction of $V$ modulo the prime ideal 
corresponding to  $q=p^r$.  The local zeta  $Z_q(u):=\exp\left(\sum_{m=1}^{\infty}
|V(\mathbf{F}_q)| \frac{u^m}{m}\right)$ is a rational function
$Z_q(u)=\prod_{i=0}^{2n}\left(P_i(u)\right)^{(-1)^{i+1}}$,
where $P_0(u)=1-u$ and $P_{2n}(u)=1-q^nu$.
Each  $P_i(u)$ is   the characteristic polynomial of the Frobenius 
endomorphism $Fr_q^i : ~(a_1,\dots, a_n)\mapsto (a_1^q,\dots, a_n^q)$ 
acting on  the $i$-th $\ell$-adic cohomology group  $H^i(V)$ of variety $V$.  
 The number of  points on $V(\mathbb{F}_q)$  is given by the Lefschetz trace formula 
$|V(\mathbb{F}_q)|=\sum_{i=0}^{2n}(-1)^i  ~tr~(Fr^i_q)$, 
where  $tr$ is the trace of  endomorphism  $Fr^i_q$, see e.g.   \cite{Ser1}. 
The Hasse-Weil zeta function of $V$ is an infinite product
$Z_V(s)=\prod_p Z_q(p^{-s}),  ~s\in\mathbf{C}$, 
where $p$ runs through all but a finite set of primes. 
A fundamental  analogy between number fields and function 
fields predicts a prime $p=\infty$  in formula $Z_q(u)=\prod_{i=0}^{2n}\left(P_i(u)\right)^{(-1)^{i+1}}$. 
 Serre constructed local factors  $\Gamma_V^i(s)$ realizing  the analogy
in terms of  the  $\Gamma$-functions attached to the Hodge structure on $V$
 \cite{Ser1}.   
 The local factors  $\Gamma_V^i(s)$  can be defined in a way similar to  the finite primes.
 Namely,  Deninger introduced  an infinite-dimensional cohomology $H^i_{ar}(V)$
and  an action of Frobenius endomorphism $Fr_{\infty}^i:   H^i_{ar}(V)\to H^i_{ar}(V)$, 
such that  $\Gamma_V^i(s)\equiv char^{-1}~Fr_{\infty}^i$, where $char ~Fr_{\infty}^i$ is   the characteristic polynomial of $Fr_{\infty}^i$
 \cite[Theorem 4.1]{Den1}.  
The following result relates the local factors at infinity with the Cuntz-Pimsner algebras.
\begin{theorem}\label{thm2.3}
{\bf (\cite[Theorem 1.1]{Nik1})}
For every smooth  $n$-dimensional 
projective variety $V$ over a number field $k$
there exist  the Cuntz-Pimsner algebras $\mathcal{O}_{A^i_{\infty}}$,  
such that the Hasse-Weil zeta function of $V$  
is given by the formula
\begin{equation}\label{eq2.6}
Z_V(s)=\prod_{i=0}^{2n} \left( char ~A_{\infty}^i \right)^{(-1)^{i+1}}.
\end{equation}
\end{theorem}

\section{Proof}
For the sake of clarity, let us outline the main ideas. 
Let $Fr^1_{\infty}$ be the Frobenius endomorphism of the 
Deninger cohomology group $H^1_{ar}(V)$. Recall that 
 $H^1_{ar}(V)\cong\mathbf{Z}[\mathbf{x}^{\pm 1}]$ is the 
 ring of Laurent polynomials in variables $\mathbf{x}=(x_1,\dots,x_{b_1})$,
 where $b_1$ is the first Betti number of variety $V$  \cite[Section 3]{Den1}. 
 The commutative ring  $A:=Fr^1_{\infty}(\mathbf{Z}[\mathbf{x}^{\pm 1}])$ is 
 the coordinate ring of a variety $V$  (Lemma \ref{lm3.1}). 
  One can consider a divisor class group $C(A)$ of  the ring $A$
 as the quotient group of divisorial ideals $D(A)$ by the subgroup $F(A)$ 
 consisting of the principal ideals. Since $A$ is a locally factorial ring, the group $C(A)$ is isomorphic to the    
 Picard group of variety $V$ (Lemma \ref{lm3.2}). 
 The rest of the proof follows from Theorem \ref{thm2.2} 
 and  the well known formula  $C(A)\cong  \mathbf{Z}^{\infty}/(1-(Fr^1_{\infty})^t)\mathbf{Z}^{\infty}$
 (Lemma \ref{lm3.3}).
 Let us pass to a detailed argument.

\begin{lemma}\label{lm3.1}
The commutative ring  $A:=Fr^1_{\infty}(\mathbf{Z}[\mathbf{x}^{\pm 1}])$ is 
 the coordinate ring of a projective variety $V$.
 \end{lemma} 
\begin{proof}
(i)  Let $b_1$ be the first Betti number of variety $V(\mathbf{C})$. 
Recall that the first Deninger's cohomology group   $H^1_{ar}(V)$
is isomorphic to the additive group of the Laurent polynomials
$\mathbf{Z}[x_1^{\pm 1},\dots, x_{b_1}^{\pm 1}]$  \cite[Section 3]{Den1},
see also Remark \ref{rmkDen}. 
Thus one gets an endomorphism $Fr^1_{\infty}:
\mathbf{Z}[x_1^{\pm 1},\dots, x_{b_1}^{\pm 1}]\to \mathbf{Z}[x_1^{\pm 1},\dots, x_{b_1}^{\pm 1}]$.
Notice that any basis of  $\mathbf{Z}[x_1^{\pm 1},\dots, x_{b_1}^{\pm 1}]$ is countably infinite, and therefore  
 $Fr^1_{\infty}$ is given by an infinite-dimensional integer matrix; hence the notation.

\bigskip
(ii) Let $V\to V'$ be an isomorphism between
projective varieties $V$ and $V'$.  The cohomology functor induces 
an isomorphism $\phi : H_{ar}^i(V)\to  H_{ar}^i(V')$ of the corresponding 
Deninger cohomology groups.  Recall that $H_{ar}^i(V)\cong \mathbf{R}[\mathbf{x}^{\pm 1}]$
and since the isomorphism of $V$ is defined over a number field $k$,  one 
gets an isomorphism $\phi:\mathbf{Z}[\mathbf{x}^{\pm 1}]\to \mathbf{Z}[\mathbf{y}^{\pm 1}]$.
The  group isomorphism $\phi$  extends to a ring  isomorphism by choice of a monomial basis in the ring
of the Laurent polynomials, and vice versa.
It remains to notice that the endomorphism $Fr^i_{\infty}: \mathbf{Z}[\mathbf{x}^{\pm 1}]
\to \mathbf{Z}[\mathbf{x}^{\pm 1}]$ commutes with $\phi$ and therefore  
$\phi(Fr^i_{\infty}(\mathbf{Z}[\mathbf{x}^{\pm 1}]))=Fr_{\infty}^i(\mathbf{Z}[\mathbf{y}^{\pm 1}])$.   
In other words, one gets a ring isomorphism $A\cong A'$.  Lemma \ref{lm3.1} is proved.
\end{proof}
\begin{remark}\label{rmkDen}
In  \cite[Section 3]{Den1} the notation $B^i_{ar}$ has been used instead of  $H^i_{ar}(V)$;
the  $B^i_{ar}$ is proved isomorphic  to the ring of Laurent polynomials in  variable $T$. 
The result extends to the multivariable  Laurent polynomials \cite[Section 2.1]{Nik1}. 
\end{remark}

\begin{lemma}\label{lm3.2}
With all the notations above
$C(A)\cong \operatorname{Pic}_c(V).$ 
\end{lemma} 
\begin{proof}
(i) Recall that   $A:=Fr^1_{\infty}(\mathbf{Z}[\mathbf{x}^{\pm 1}])$ is a cluster algebra;
we refer the reader to \cite[Section 2.3]{Nik1} for the definition and proof of this fact.  Such algebras  are  Krull domains
 and, therefore, one can consider the 
corresponding  divisor class group $C(A)$  \cite{ElsLamSme1}.

\bigskip
(ii) Let $\operatorname{Pic} ~(A)$ be the Picard group of the ring $A$, i.e. a subgroup of $C(A)$ generated 
by the invertible ideals of $A$.  Since $A$ is a locally factorial ring, one gets 
a ring isomorphism $\operatorname{Pic} ~(A)\cong C(A)$.

\bigskip
(iii) It is known  that $\operatorname{Pic} ~(A)\cong \operatorname{Pic}~(\operatorname{Spec}~A)$, where $\operatorname{Spec}~A\cong V$. 
On the other hand, we recall that the ring $A:=Fr^1_{\infty}(\mathbf{Z}[\mathbf{x}^{\pm 1}])$ 
comes  from the
action of Frobenius endomorphism $Fr^1_{\infty}$ on the Deninger cohomology
group $H^1_{ar}(V)$ of variety $V$. 
Since  the local factor at infinity is given 
by  the  characteristic polynomial of the linear action of  $Fr^1_{\infty}$ 
on  $H^1_{ar}(V)$,  we conclude that the Picard group of the ring $A$
encodes the data of the arithmetic scheme $V\to\operatorname{Spec}~A$ at the prime $p=\infty$. 
In other words, one gets an isomorphism:
\begin{equation}\label{eq3.0}
\operatorname{Pic}~(\operatorname{Spec}~A)\cong \operatorname{Pic}_c(V).
\end{equation}

\bigskip
(iv)  One gets following isomorphisms of the abelian groups:
\begin{equation}\label{concl}
 C(A)\cong \operatorname{Pic} ~(A)\cong \operatorname{Pic}~(\operatorname{Spec}~A)\cong \operatorname{Pic}_c(V).
 \end{equation}
 
 It follows from (\ref{concl}) that   $C(A)\cong \operatorname{Pic}_c(V).$
Lemma \ref{lm3.2} is proved.
\end{proof}

\begin{lemma}\label{lm3.3}
$K_0(\mathcal{O}_{A^1_{\infty}})\cong \operatorname{Pic}_c(V).$ 
\end{lemma} 
\begin{proof}
(i) Let us calculate $C(A)$ in terms of the Frobenius endomorphisms 
$Fr_{\infty}^i: \mathbf{Z}[\mathbf{x}^{\pm 1}]\to \mathbf{Z}[\mathbf{x}^{\pm 1}]$. 
Notice  that ideals of $A$ of height $i$ correspond to Frobenius 
endomorphisms  $Fr_{\infty}^i$ for all  $0\le i\le 2n$. 
Since $C(A)$ is generated by the classes of the prime ideals of height $1$,
we conclude that $C(A)$ depends only on the endomorphism 
$Fr_{\infty}^1: \mathbf{Z}[\mathbf{x}^{\pm 1}]\to \mathbf{Z}[\mathbf{x}^{\pm 1}]$.

\bigskip
(ii) Consider the additive group $L\cong\mathbf{Z}^{\infty}$ of the ring $\mathbf{Z}[\mathbf{x}^{\pm 1}]$. 
The endomorphism $Fr_{\infty}^1$ defines a sub-lattice $L'$ of the fixed points of $L$
whose cosets form a group:
\begin{equation}\label{eq3.1}
 \frac{\mathbf{Z}^{\infty}}{(1-(Fr_{\infty}^1))^t \mathbf{Z}^{\infty}}
\end{equation}

\bigskip
(iii) On the other hand,  $C(A)$ is the quotient of divisorial ideals $D(A)$
by the principal ideals $F(A)$  of the ring $A$. 
It is not hard to see,  that each principal ideal of $A$  can be written as $xL$ for some $x\in L'$.
In other words, $F(A)=\{xL ~|~ x\in L'\}$. Thus  
\begin{equation}\label{eq3.2}
C(A)=D(A)/F(A)\cong L/L'\cong \frac{\mathbf{Z}^{\infty}}{(1-(Fr_{\infty}^1)^t) \mathbf{Z}^{\infty}}.
\end{equation}

\bigskip
(iv)  Likewise, we set $A_{\infty}=Fr^1_{\infty}$ in Theorem \ref{thm2.2} and we get
\begin{equation}\label{eq3.3}
K_0(\mathcal{O}_{A^1_{\infty}}) \cong \frac{\mathbf{Z}^{\infty}}{(1-(Fr_{\infty}^1)^t) \mathbf{Z}^{\infty}}.
\end{equation}

\bigskip
(v) Comparing (\ref{eq3.2}) and (\ref{eq3.3}) we conclude that $K_0(\mathcal{O}_{A^1_{\infty}}) \cong C(A)$. 
But Lemma \ref{lm3.2} says that $C(A)\cong \operatorname{Pic}_c(V)$. Therefore   $K_0(\mathcal{O}_{A^1_{\infty}}) \cong \operatorname{Pic}_c(V)$.
Lemma \ref{lm3.3} is proved. 
\end{proof}

\bigskip
Theorem \ref{thm1.1} follows from Lemma \ref{lm3.3}.

\section{Remarks}
Recall that the $K$-theory of a Cuntz-Krieger algebra  $\mathcal{O}_{A}$  is given 
by the formulas:
\begin{equation}\label{eq4.1}
\left\{
\begin{array}{ccc}
K_0(\mathcal{O}_{A}) &\cong&  \mathbf{Z}^n / (1-A^t) \mathbf{Z}^n,\\
&&\\
K_1(\mathcal{O}_{A}) &\cong& Ker ~(1-A^t).
\end{array}
\right.
\end{equation}
It is immediate  that $|K_0(\mathcal{O}_{A})|<\infty$ if and only if
$\det ~(1-A^t)\ne 0$.  Whenever $\det ~(1-A^t)=0$ both  $K_0(\mathcal{O}_{A})$ and $K_1(\mathcal{O}_{A})$ are 
infinite groups  \cite[Exercise 10.11.9 (c)]{B}.  
On the other hand, it is known that $|K_0(\mathcal{O}_{A})|=|V(\mathbf{F}_p)|$ for a matrix $A\in\operatorname{GL}_n(\mathbf{Z}),$
see e.g. \cite[Section 1]{Nik1}.

It is useful to think of  $K_0(\mathcal{O}_{A_{\infty}})$ as the inductive limit of finite abelian groups $K_0(\mathcal{O}_{A_i})$,
see definition of the AF-core  $\mathscr{F}$ in Section 2.2  and formula (\ref{eq2.4}).  
In other words, 
\begin{equation}\label{eq4.2}
\lim_{i\to\infty} K_0(\mathcal{O}_{A_i}) \cong  K_0(\mathcal{O}_{A_{\infty}}).
\end{equation}
However, if $\det~(1-A_m)=0$ for some $m$, one gets $|K_0(\mathcal{O}_{A_m})|=\infty$. 
In this case the convergent sequence (\ref{eq4.2}) must terminate
after a finite number of steps. Given that finiteness of the $K_0$-groups is related 
to such of the points of variety $V$, we have the following conjecture.
\begin{conjecture}\label{cnj4.1}
If $V(k)$ is a projective variety over the number field $k$, then
\begin{equation}\label{eq4.3}
|V(k)|=\begin{cases}
\hbox{finite}, &
\hbox{if} \quad \det ~(1-A^1_{\infty})=0,\cr 
\hbox{infinite}, & \hbox{otherwise.}
\end{cases}
\end{equation}
\end{conjecture}

\section*{Data availability}
  
  Data sharing not applicable to this article as no datasets were generated or analyzed during the current study.
   
\section*{Conflict of interest}
On behalf of all co-authors, the corresponding author states that there is no conflict of interest.
  

\section*{Funding declaration}
The author was partly supported by the NSF-CBMS grant 2430454.

\begin{acknowledgement}
  The author would like to thank the anonymous referee who provided thoughtful  comments on an earlier version of the manuscript.
\end{acknowledgement}


\end{document}